\documentclass[leqno,11pt]{amsart}

\usepackage{geometry}

\usepackage[all]{xy} % diagrams
\usepackage{amsmath, amssymb, amsfonts, latexsym, mdwlist, amsthm, amscd}
\usepackage{subfig}
\usepackage{graphicx}
\usepackage{wrapfig}

\usepackage[bookmarks, colorlinks, breaklinks, pdftitle={Stable vector bundles on Enriques surfaces},
pdfauthor={Howard Nuer}]{hyperref}
\hypersetup{linkcolor=blue,citecolor=blue,filecolor=black,urlcolor=blue}

\usepackage{tikz}
\usetikzlibrary{calc,trees,positioning,arrows,chains,shapes.geometric,%
    decorations.pathreplacing,decorations.pathmorphing,shapes,%
    matrix,shapes.symbols}

\tikzset{
>=stealth',
  punktchain/.style={
    rectangle,
    rounded corners,
    draw=black, thick,
    minimum height=3em,
    text centered,
    on chain},
  line/.style={draw, thick, <-},
  element/.style={
    tape,
    top color=white,
    bottom color=blue!50!black!60!,
    minimum width=8em,
    draw=blue!40!black!90, very thick,
    text width=10em,
    minimum height=3.5em,
    text centered,
    on chain},
  every join/.style={->, thick,shorten >=1pt},
  decoration={brace},
  tuborg/.style={decorate},
  tubnode/.style={midway, right=2pt},
}

\usepackage{paralist}
\setdefaultenum{(a)}{(i)}{}{}
\usepackage{enumitem} % for space-saving description environments}

\def\C{\ensuremath{\mathbb{C}}}

\def\P{\ensuremath{\mathbb{P}}}
\def\Q{\ensuremath{\mathbb{Q}}}
\def\R{\ensuremath{\mathbb{R}}}
\def\Z{\ensuremath{\mathbb{Z}}}

\def\alg{\mathrm{alg}}
\def\Amp{\mathrm{Amp}}

\def\ch{\mathop{\mathrm{ch}}\nolimits}

\def\Coh{\mathop{\mathrm{Coh}}\nolimits}

\def\deg{\mathop{\mathrm{deg}}}

\def\dim{\mathop{\mathrm{dim}}\nolimits}

\def\inf{\mathop{\mathrm{inf}}\nolimits}

\def\Ext{\mathop{\mathrm{Ext}}\nolimits}
\def\ext{\mathop{\mathrm{ext}}\nolimits}
\def\lExt{\mathop{\mathcal Ext}\nolimits} % means local Ext

\def\Hal{H^*_{\alg}}
\def\Hilb{\mathop{\mathrm{Hilb}}\nolimits}
\def\fhilb{\mathop{\Hilb_{D,i_1,...,i_s}}\nolimits}
\def\HN{\mathop{\mathrm{HN}}\nolimits}
\def\Hom{\mathop{\mathrm{Hom}}\nolimits}

\def\im{\mathop{\mathrm{im}}\nolimits}

\def\mod{\mathop{\mathrm{mod}}\nolimits}
\def\min{\mathop{\mathrm{min}}\nolimits}

\def\Nef{\mathrm{Nef}}
\def\num{\mathop{\mathrm{num}}\nolimits}
\def\Num{\mathop{\mathrm{Num}}\nolimits}
\def\NS{\mathop{\mathrm{NS}}\nolimits}

\def\Pic{\mathop{\mathrm{Pic}}}

\def\rk{\mathop{\mathrm{rk}}}

\def\SSL{\mathop{\mathrm{SSL}}}

\def\td{\mathop{\mathrm{td}}\nolimits}

\def\v{\mathop{\pi^*v}\nolimits}

\def\MG13{\ensuremath{{\mathcal M}_{\Gamma_1(3)}}}
\def\tildeMG13{\ensuremath{\widetilde{\mathcal M}_{\Gamma_1(3)}}}

\def\into{\ensuremath{\hookrightarrow}}

\def\blank{\underline{\hphantom{A}}}

\makeatletter
\newtheorem*{rep@theorem}{\rep@title}
\newcommand{\newreptheorem}[2]{%
\newenvironment{rep#1}[1]{%
 \def\rep@title{#2 \ref{##1}}%
 \begin{rep@theorem}}%
 {\end{rep@theorem}}}
\makeatother

\newtheorem{Thm}{Theorem}[section]
\newreptheorem{Thm}{Theorem}
\newtheorem{Prop}[Thm]{Proposition}

\newtheorem{Lem}[Thm]{Lemma}

\newreptheorem{Cor}{Corollary}

\newreptheorem{Con}{Conjecture}

\newtheorem{thm-int}{Theorem}

\theoremstyle{definition}
\newtheorem{Def-s}[Thm]{Definition}
\newtheorem{Def}[Thm]{Definition}
\newtheorem{Rem}[Thm]{Remark}

\def\C{\ensuremath{\mathbb{C}}}

\def\P{\ensuremath{\mathbb{P}}}
\def\Q{\ensuremath{\mathbb{Q}}}
\def\R{\ensuremath{\mathbb{R}}}
\def\Z{\ensuremath{\mathbb{Z}}}

\def\EE{\ensuremath{\mathcal E}}

\def\NN{\ensuremath{\mathcal N}}
\def\OO{\ensuremath{\mathcal O}}
\def\PP{\ensuremath{\mathcal P}}

\def\Y{\ensuremath{\tilde{Y}}}

\def\MMM{\mathfrak M}

\newcommand{\ignore}[1]{}

\begin{document}

\title[Stable vector bundles on Enriques surfaces]{A note on the existence of stable vector bundles on Enriques surfaces}

\author{Howard Nuer}
\address{Department of Mathematics, Rutgers University, 110 Frelinghuysen Rd., Piscataway, NJ 08854, USA}
\email{hjn11@math.rutgers.edu}
\urladdr{http://math.rutgers.edu/~hjn11/}



\begin{abstract}
We prove the non-emptiness of $M_{H,Y}(v)$, the moduli space of Gieseker-semistable sheaves on an unnodal Enriques surface $Y$ with Mukai vector $v$ of positive rank with respect to a generic polarization $H$.  This completes the chain of progress initiated by H. Kim in \cite{Kim98}.  We also show that the stable locus $M^s_{H,Y}(v)\neq\varnothing$ for $v^2>0$.  Finally, we prove irreducibility of $M_{H,Y}(v)$ in case $v^2=0$ and $v$ primitive. 
\end{abstract}

\vspace{-1em}

\maketitle

\setcounter{tocdepth}{1}
\tableofcontents

\section{Introduction}\label{sec:intro}
\subsection{Background}
Over the last few decades there has been a great deal of interest in the study of moduli spaces of coherent sheaves on smooth projective varieties, often inspired by mathematical physics and gauge theory.  In order to construct such a moduli space, in particular to obtain boundedness, one must restrict one's attention to coherent sheaves satisfying some sense of stability with fixed topological invariants, encoded in the Mukai vector $v$.  The two most ubiquitous definitions of stability are $\mu_H$-stability, or slope stability, and Gieseker-stability, both of which are defined by choosing an ample polarization $H$ on the base variety $X$.  Among the many fascinating aspects of these moduli spaces, other than their uses in physics, is the intimate connection they have with the underlying projective variety.  

A particularly tight connection with $X$ is via the choice of polarization $H\in\Amp(X)$.  As the moduli spaces $M_{H,X}(v)$ of Gieseker-semistable sheaves on $X$ with Mukai vector $v$ is constructed as a GIT (Geometric Invariant Theory) quotient with respect to $H$, varying the polarization $H$ induces a VGIT (Variation of GIT) birational transformation (as defined and studied in \cite{DH98} and \cite{Tha96}).  The corresponding connection with the birational geometry of Gieseker moduli spaces has been studied in \cite{MW97}.  Most important here is that the other birational models obtained from varying the polarization are moduli spaces as well, albeit with a slightly different moduli functor.  The so-called Mumford-Thaddeus principle studied in \cite{MW97} can be seen as an extension of the Hassett-Keel program for $\overline{M}_{g,n}$, where the minimal models of $\overline{M}_{g,n}$ obtained by running the MMP are hoped to be modular themselves.

Possibly the most studied case of these Gieseker moduli spaces is when $X$ is a smooth projective surface, and here a second important connection between the geometry of $M_{H,X}(v)$ and that of $X$ emerges.  For example, when $X$ is a projective $K3$ surface and $v$ is primitive, the moduli spaces $M_{H,X}(v)$ of Gieseker-stable sheaves on $X$ of Mukai vector $v$ are projective hyperk\"{a}hler (i.e.  irreducible holomorphic symplectic) manifolds \cite{Muk84}.  These are incredibly rare varieties with a beautiful and rigid geometry, and they are quite important as one of the building blocks of varieties with trivial $c_1$ \cite{Bea83}.  Along with the related case when $v$ is two times a primitive Mukai vector, these Gieseker moduli spaces (or their smooth resolutions in this non-primitive case) form all but two of the known deformation equivalence classes of such varieties.  The other two come instead from constructions involving Gieseker moduli spaces on Abelian surfaces.  Similarly, the Gieseker moduli spaces on rational surfaces of various types have been studied and interesting connections with the underlying surface have been unveiled.

While moduli of sheaves on $K3$ and Abelian surfaces have arguably received the most attention, the corresponding moduli spaces on the related Enriques surfaces have been much less studied.  Recall that an \emph{Enriques surface} is a smooth projective surface $Y$ with $h^1(\OO_Y)=0$ and canonical bundle $\OO(K_Y)$ a non-trivial 2-torsion element of $\Pic(Y)$.\footnote{We work over $\C$ in this paper, so this is the correct definition.}  The investigation of Gieseker moduli spaces on an Enriques surface $Y$ was started by H. Kim in \cite{Kim}, where he proved some general structure results about the locus parametrizing stable locally free sheaves, but in general, the picture is much less complete.  

Kim himself continued his investigation in \cite{Kim06} with existence results for locally free sheaves in rank 2 and a description of some of the geometry of these moduli spaces.  K. Yoshioka succesfully computed in \cite{Yos03} the Hodge polynomial of $M_{H,Y}(v)$ for primitive $v$ of positive odd rank and generic polarization $H$ on an unnodal $Y$, showing that it is equal to the Hodge polynomial of the Hilbert scheme of $\frac{v^2+1}{2}$ points on $Y$.  It follows that these moduli spaces are non-empty, consisting of two isomorphic irreducible connected components parametrizing sheaves whose determinant line bundles differ by $K_Y$.  M. Hauzer \cite{Hau10} refined the techniques of Yoshioka to conclude that the Hodge polynomials in even rank are the same as those of a moduli space of sheaves with rank 2 or 4.  Moreover, like Enriques surfaces themselves, K. Yamada \cite{Yamada} proved that under certain minor assumptions $M_{H,Y}(v)$ has torsion canonical divisor.  Finally, G. Sacc\`{a} \cite{Sac} obtained some beautiful results about the geometry of Gieseker moduli spaces in rank 0.  Nevertheless, even non-emptiness and irreducibility for moduli of sheaves of even rank is unknown in general.

In this paper, we complete the sequence of results mentioned above for Gieseker moduli spaces of torsion-free sheaves with regard to non-emptiness.

\subsection{Summary of Main Results}
The central result of this paper is the non-emptiness of the moduli space $M_{H,Y}(v)$ parametrizing Gieseker-semistable sheaves of positive rank on an unnodal Enriques surface with respect to a generic parametrization $H$:
\begin{Thm} Let $v_0$ be a primitive Mukai vector of positive rank with $v_0^2\geq -1$ and $m\in\Z_{>0}$.  For generic polarization $H$, the moduli space of Gieseker-semistable sheaves of Mukai vector $v=mv_0$ with respect to $H$ on an unnodal Enriques surface $Y$, $M_{H,Y}(v)$, is non-empty.
\end{Thm}

We prove this theorem in Section 6 by first reducing to the case $m=1$ so that every Gieseker-semistable sheaf is stable as well. The result for $v_0$ of odd rank was proved in \cite{Yos03}.  From Hauzer's main theorem \cite[Theorem 0.2]{Hau10}, the question of the existence and irreducibility of Gieseker moduli spaces in positive even rank is reduced to the same question in ranks 2 and 4.  The rank 2 case was essentially solved in \cite{Kim06}, and we prove existence in rank 4 by reducing to the following theorem:

\begin{Thm}Let $v=(4,c_1,2-k)$ be a primitive Mukai vector, with $-5\leq k\leq 1$.  Suppose that $c_1$ is ample and denote it by $H$.  Assume that $c_1^2$ satisfies the inequality in (\ref{table}) for the corresponding $k$.  Then there exists a $\mu_H$-stable vector bundle with Mukai vector $v$ on an unnodal Enriques surface $Y$.
\end{Thm}

The proof of this theorem in Section 5 uses the by-now classical Serre correspondence as in the proof of \cite[Theorem 5.1.6]{HL}, but instead of the usual techniques which demand $c_2$ to be sufficiently large, we use the large Picard lattice of Enriques surfaces to reduce ourselves to the assumptions of the theorem.  More notably, we show that these moduli spaces contain open subsets parametrizing $\mu$-stable \emph{vector bundles}, results which are again usually only achieved by asymptotic assumptions on $c_2$.  The reduction to primitive Mukai vectors of the form $(4,c_1,2-k)$ for $-5\leq k\leq 1$ is achieved in Section 4 by generalizing the methods of \cite{Kim98} to rank 4 vector bundles after first reviewing the necessary facts about divisors on Enriques surfaces in Section 3.  

We sum all of this up with the following description of the global and local properties of $M_{H,Y}(v)$ in Section 7:
\begin{Thm} Let $v=mv_0$ be a Mukai vector with $v_0$ primitive and $m>0$ with $H$ generic with respect to $v$.  Assume that $Y$ is unnodal.
\begin{enumerate}
\item The moduli space of Gieseker-semistable sheaves $M_{H,Y}(v)\neq \varnothing$ if and only if $v_0^2\geq -1$.
\item Either $\dim M_{H,Y}(v)=v^2+1$ and $M^s_{H,Y}(v)\neq\varnothing$,  or $m>1$ and $v_0^2\leq 0$.
\item If $M_{H,Y}(v)\neq M^s_{H,Y}(v)$ and $M^s_{H,Y}(v)\neq\varnothing$, the codimension of the semistable locus is at least 2 if and only if $v_0^2>1$ or $m>2$.  Moreover, in this case $M_{H,Y}(v)$ is normal with torsion canonical divisor.
\end{enumerate}
\end{Thm}

We finish the paper in Section 8 with a proof of irreducibility in case $v$ is primitive and $v^2=0$:

\begin{Thm} Let $v$ be a primitive Mukai vector on an unnodal Enriques surface $Y$ with $v^2=0$.  Then $M_{H,Y}(v,L)$ is irreducible, where $L$ is one of the two lines bundles with $c_1(L)=c_1(v)$.
\end{Thm} 

Note that $0=v^2=c_1^2-2rs$, so the evenness of $c_1^2$ implies that either $r$ or $2s$ is even.  But since $2s=r+c_1^2-2c_2$ implies that $r\equiv 2s(\mod 2)$, both $r$ and $2s$ are even, so this result does not follow from \cite{Yos03}.

\subsection{Acknowledgements} The author would like to thank his advisor, Lev Borisov,  for his constant support and guidance, especially in talking out ideas.  The author was partially supported by NSF grant DSM 1201466.
\section{Review: Moduli spaces of semistable sheaves}
We start by recalling the basic lattice-theoretical structure, given by the Mukai lattice.
We then review slope and Gieseker stability and the known results about existence and non-emptiness of moduli spaces of semistable sheaves.

\subsection*{The algebraic Mukai lattice}
Let $Y$ be an Enriques surface.  Its universal cover is a projective $K3$ surface $\tilde{Y}$.  $\tilde{Y}$ comes with a fixed-point free covering involution $\iota$ such that $Y=\tilde{Y}/\langle \iota\rangle$.

Denote by \begin{equation}\label{Mukai lattice} H^*_{\alg}(Y,\Z):=H^0(Y,\Z)\oplus \mathrm{NS}(Y)\oplus \frac{1}{2}\Z \rho_Y\subset H^0(Y,\Z)\oplus H^2(Y,\Z)\oplus H^4(Y,\Q),\end{equation} the Mukai lattice of $Y$, where $\rho_Y$ is the fundamental class.  Then $$v \colon K_{\num}({Y}) \xrightarrow{\sim} H^*_{\alg}({Y}, \Z)$$ given by $v(E) = \ch(E) \sqrt{\td({Y})}$ associates to $E$ in the numerical $K$-group of $Y$ its associated Mukai vector.  Written according to the decomposition \eqref{Mukai lattice} this becomes \[v(E)=(r(E),c_1(E),\frac{r(E)}{2}+\ch_2(E)).\]
We denote the Mukai pairing $H^*_{\alg}(Y,\Z)\times H^*_{\alg}(Y,\Z)\to \Z$ by $(\blank, \blank)$; it can be
defined by $(v(E), v(F)) := - \chi(E, F)$, where $$\chi(E,F)=\sum_p (-1)^p\ext^p(E,F)$$ denotes the Euler pairing on $K(Y)$.  This becomes non-degenerate when modding out by its kernel to get $K_{\num}(Y)$.  
According to the decomposition \eqref{Mukai lattice}, we have
\[
\left( (r,c,s),(r',c',s')\right) = c.c' - rs' - r's,
\]
for $(r,c,s),(r',c',s')\in H^*_{\alg}(Y,\Z)$.  

We call a Mukai vector $v$ \emph{primitive} if it is not divisible in $H^*_\alg(Y,\Z)$.  Note that the covering space map $\pi:\Y\rightarrow Y$ induces a primitive embedding $$\pi^*:\Hal(Y,\Z)\into\Hal(\Y,\Z):=H^0(\tilde{Y},\Z)\oplus \NS(\tilde{Y})\oplus H^4(\tilde{Y},\Z)$$ such that $(\v,\pi^*w)=2(v,w)$, and it identifies $\Hal(Y,\Z)$ with the $\iota^*$-invariant component of $\Hal(\tilde{Y},\Z)$.  It follows that for a primitive Mukai vector $v$, $\v$ is primitive as well.

\subsection*{Slope stability}
Let $H\in\Amp(X)$ on a smooth projective surface $X$.
We define the slope function
$\mu_H$ on $\Coh X$ by
\begin{equation} \label{eq:muomegabeta}
\mu_H(E) = 
\begin{cases}
\frac{H.c_1(E)}{r(E)} & \text{if $r(E) > 0$,} \\
+\infty & \text{if $r(E) = 0$.}
\end{cases}
\end{equation}
This gives a notion of slope stability for sheaves, for which Harder-Narasimhan filtrations exist (see \cite[Section 1.6]{HL}).  Recall that a torsion free coherent sheaf $E$ is called slope \emph{semistable} (resp. \emph{stable}) with respect to $H$ if for every $F\subset E$ with $0<r(F)<r(E)$ we have $\mu_H(F)\leq\mu_H(E)$ (resp. $\mu_H(F)<\mu_H(E)$).  
We will sometimes use the notation $\mu$-stability, or $\mu_H$-stability if we want to make the dependence on $H$ clear.  Also recall that every torsion free coherent sheaf $E$ admits a unique \emph{Harder-Narasimhan filtration} $$0=\HN^0(E)\subset\HN^1(E)\subset ... \subset\HN^n(E)=E,$$ with factors $E_i=\HN^i(E)/\HN^{i-1}(E)$ satisfying $$\mu_H(E_1)>...>\mu_H(E_n).$$  

\subsection*{Gieseker stability}
Let $H\in\Amp(X)$ on a smooth projective surface $X$.
Recall that the \emph{Hilbert polynomial} is defined by
\[
P(E,m) := \int_X (1,mH,\frac{m^2H^2}{2}) . v(E),
\]
for $E\in\Coh(X)$.  This polynomial can uniquely be written in the form $\sum_{i=0}^{\dim(E)} a_i(E) \frac{m^i}{i!}$, and we define the \emph{reduced Hilbert polynomial} by 
$$p(E,M):=\frac{P(E,m)}{a_{\dim(E)}(E)}.$$
This gives rise to the notion of $H$-Gieseker stability for sheaves.
We refer to \cite[Section 1]{HL} for basic properties of Gieseker stability, but we just mention that like above, a pure dimensional sheaf $E$ is called \emph{Gieseker semistable} (resp. \emph{Gieseker stable}) if for every proper subobject $0\neq F\subset E$, $p(F,m)\leq p(E,m)$ (resp. $p(F,m)<p(E,m)$) for all $m\gg 0$.  Harder-Narasimhan filtrations are defined analogously as above with $p$ replacing $\mu_H$, and in addition, every semistable sheaf admits a \emph{Jordan-H\"{o}lder filtration} with stable factors all of the same reduced Hilbert polynomial.  The filtration is not unique, but the factors and their multiplicities are.

It is worth pointing out that $$E\text{ is }\mu-\text{stable}\Rightarrow E\text{ is Gieseker-stable }\Rightarrow E\text{ is Gieseker-semistable }\Rightarrow E\text{ is }\mu-\text{semistable}.$$

\subsection*{Moduli spaces of stable sheaves}
Let $H\in\Amp(X)$ on a smooth projective surface $X$.
We fix a Mukai vector $v\in H^*_{\alg}(X,\Z)$ (or in other words, we fix the topological invariants $r,c_1,c_2$).
We denote by $\MMM_{H,X}(v)$ the moduli stack of flat families of $H$-Gieseker semistable sheaves with Mukai vector $v$.
By \cite[Chapter 4]{HL} there exists a projective variety $M_{H,X}(v)$ which is a coarse moduli space parameterizing $S$-equivalence classes of semistable sheaves.  Recall that two semistable sheaves are called $S$-equivalent if they have the same Jordan-H\"{o}lder factors with the same multiplicities.
The open substack $\MMM_{H,X}^s(v)\subseteq \MMM_{H,X}(v)$ parameterizing stable sheaves is a $\mathbb{G}_m$-gerbe over the similarly defined open subset $M_{H,X}^s(v)\subseteq M_{H,X}(v)$ of the coarse moduli space.

We must also recall the definition of a polarization that is \emph{generic} with respect to a given Mukai vector $v$ satisfying $v^2>-r(v)^2$ and $r(v)\geq 2$ (see \cite[Section 4.C]{HL}).  Consider $\xi\in\Num(X)$ with $-\frac{r(v)^2}{4}(v^2+r(v)^2)\leq \xi^2<0$.  The \emph{wall} for $v$ corresponding to $\xi$ is the real hyperplane $\xi^{\perp}\subset\Nef(X)$.  These walls are locally finite.  A polarization $H\in\Amp(X)$ is \emph{generic} with respect to $v$ if it does not lie on any of these walls.  An important consequence of this is that for a destabilizing subobject $F$ of $E$, with $v(E)$ as above and $H$ generic with respect to $v(E)$, $v(F)\in\R_{>0}v(E)$.  So if $v$ is primitive, any $H$-Gieseker semistable sheaf $E$ with $v(E)=v$ is Gieseker-stable as well.  If, in addition, $c_1$ is primitive in $\NS(X)$, then any $\mu_H$-semistable sheaf is even $\mu_H$-stable. 

We review one last facet of the theory of moduli spaces of sheaves.  While a coarse moduli space exists, it is not always a fine moduli space, i.e. there does not always exist a universal family of semistable sheaves.  To remedy the possible lack of a universal family, Mukai \cite{Muk87} came up with the following substitute, which is usually good enough for most purposes:

\begin{Def} Let $T$ be an algebraic space of finite-type over $\C$ and $X$ a smooth projective variety.
\begin{enumerate} 
\item A family $\EE$ on $T\times X$ is called a \emph{quasi-family} of objects in $\MMM_{H,X}(v)$ if for all closed points $t\in T$, there exists $E\in \MMM_{H,X}(v)(\C)$ such that $\EE_t\cong E^{\oplus \rho}$, where $\rho>0$ is an integer which is called the \emph{similitude} and is locally constant on $T$.
\item Two quasi-families $\EE$ and $\EE'$ on $T$, of similitudes $\rho$ and $\rho'$, respectively, are called \emph{equivalent} if there are locally free sheaves $\NN$ and $\NN'$ on $T$ such that $\EE\otimes p_T^*\NN\cong \EE'\otimes p_T^*\NN'$.  It follows that the similitudes are related by $\rk \NN \cdot \rho=\rk \NN'\cdot \rho'$.
\item A quasi-family $\EE$ is called \emph{quasi-universal} if for every scheme $T'$ and quasi-family $\EE'$ on $T'$, there exists a unique morphism $f:T'\to T$ such that $f^*\EE$ is equivalent to $\EE'$.
\end{enumerate}
\end{Def}

The usual techniques (see for example \cite[Theorem A.5]{Muk87} or \cite[Section 4.6]{HL}) show that a quasi-universal family exists on $M_{H,X}^s(v)$ and is unique up to equivalence.  

\subsection{Moduli of sheaves on Enriques surfaces}
Now let us recall the relevant results that are known for Enriques surfaces.  From standard results in the deformation theory of sheaves, we have for a Schur sheaf $E$, i.e. $\hom(E,E)=1$, $$v^2+1\leq \dim_E M_{H,Y}(v)\leq \dim T_E M_{H,Y}(v)=v^2+1+\ext^2(E,E)=v^2+1+\hom(E,E(K_Y)).$$  Kim's main structure result from \cite{Kim} is the following:
\begin{Thm}\label{Kim} Let $Y$ be an Enriques surface with K3 cover $\tilde{Y}$.  
\begin{enumerate}
\item $M_{H,Y}^s(v)$ is singular at $E$ if and only if $E\cong E(K_Y)$ except if $E$ belongs to a 0-dimensional component ($E$ is a spherical sheaf and $v^2=-2$) or a 2-dimensional component ($v^2=0$), along which all sheaves are fixed by $- \otimes \OO(K_Y)$.
\item The singular locus of $M_{H,Y}^s$ is a union of the images under $\pi_*$ of finitely many open subsets $M^{s,\circ}_{H,\tilde{Y}}(w)$ for different $w\in\Hal(\tilde{Y})$ such that $\pi_*(w)=v$, where $$M^{s,\circ}_{H,\tilde{Y}}(w):=\{F\in M^s_{H,\tilde{Y}}(w)|F\ncong \iota^*F\}.$$  The singular locus has even dimension at most $\frac{1}{2}(v^2+4)$.  In particular, $M_{H,Y}^s(v)$ is generically smooth and everywhere smooth if $r(v)$ is odd.
\item The pull-back map $\pi^*:M^s_{H,Y}(v)\to M_{H,\tilde{Y}}(\pi^*v)$ is a double cover onto a Lagrangian subvariety of $M_{H,\tilde{Y}}(\pi^*v)$, the fixed locus of $\iota^*$, and is branched precisely along the locus where $E\cong E(K_Y)$.
\end{enumerate}
\end{Thm}

Now recall that for a variety $X$ over $\C$, the cohomology with compact support $H^*_c(X,\Q)$ has a natural mixed Hodge structure.  Let $e^{p,q}(X):=\sum_k (-1)^k h^{p,q}(H^k_c(X))$ and $e(X):=\sum_{p,q}e^{p,q}(X)x^p y^q$ be the virtual Hodge number and Hodge polynomial, respectively.  Moreover, for an Enriques surface $Y$ we recall that the kernel of $\NS(Y)\to \Num(Y)$ is given by $\langle K_Y\rangle$, and thus $$M_{H,Y}(v)=M_{H,Y}(v,L_1)\coprod M_{H,Y}(v,L_2),$$ where $M_{H,Y}(v,L_i)$ denotes those $E\in M_{H,Y}(v)$  with $\det(E)=L_i$ and $L_2=L_1(K_Y)\in \Pic(Y)$ so $c_1=c_1(L_1)=c_2(L_2)\in \Num(Y)$.  Let us recall the following definition:

\begin{Def} A smooth projective surface $X$ is called \emph{unnodal} if it contains no curves of negative self-intersection and \emph{nodal} otherwise, where we reiterate that this does not indicate the presence of singularities in this paper.
\end{Def}

An Enriques surface $Y$ is thus unnodal if it contains no $(-2)$-curves, and these are generic in their moduli space.  An important consequence of this is that the ample cone is entirely round, i.e. $D\in\Pic(Y)$ is ample if and only if $D^2>0$ and it intersects some effective curve positively.  Moreover, $\iota^*$ acts trivially on $\Pic(\tilde{Y})$ in this case, so $\pi^*\Pic(Y)=\Pic(\tilde{Y})$.  An important consequence of this is that for any $E$ with $E\cong E(K_Y)$, $v(E)$ must be divisible by 2 in $\Hal(Y,\Z)$.  Thus for primitive $v$, $M^s_{H,Y}(v)$ is smooth of dimension $v^2+1$.  The following result is proved in \cite{Yos03}:

\begin{Thm}\label{yosh odd} Let $v=(r,c_1,s)\in H^*_{\alg}(Y,\Z)$ be a primitive Mukai vector with $r$ odd and $Y$ unnodal.  Then \[e(M_{H,Y}(v,L))=e(Y^{[\frac{v^2+1}{2}]}),\] for a generic $H$, where $L\in \Pic(Y)$ satisfies $c_1(L)=c_1$.  In particular, 
\begin{itemize}
\item $M_{H,Y}(v)\neq \varnothing$ for a generic $H$ if and only if $v^2\geq -1$.
\item $M_{H,Y}(v,L)$ is irreducible for generic $H$.
\end{itemize}  
\end{Thm}

For even rank Mukai vectors, Hauzer proved the following in \cite{Hau10}:
\begin{Thm}\label{hauz even} Let $Y$ be an unnodal Enriques surface and $v=(r,c_1,s)\in H^*_\alg(Y,\Z)$ a primitive Mukai vector with $r$ even.  Then for generic polarization $H$ we have \[e(M_{H,Y}(v,L))=e(M_{H,Y}((r',c_1',-s'/2),L')),\] where $r'$ is 2 or 4.
\end{Thm}

Non-emptiness of $M_{H,Y}(v)$ with $v^2\geq 1$ was essentially proved in \cite{Kim06} for the case $r(v)=2$.  He proved irreducibility in half of the cases.  

\section{Review: Divisors on Enriques surfaces}
To aid in our investigation of stable rank 4 bundles, we first reduce ourselves to a finite number of cases we must consider.  Since stability is unchanged upon twisting by a line bundle, we are able to reduce to the case when there is a specific relationship between $c_2$ and $c_1^2$.  But first let us recall some facts about divisors on an unnodal Enriques surface $Y$.

To begin with, recall that for an Enriques surface $Y$, $$\Pic(Y)\cong U\oplus E_8\oplus \langle K_Y\rangle,$$ where $U\cong \left(\begin{matrix} 0&1\\1&0\end{matrix}\right)$ is the hyperbolic lattice and $E_8$ is the even positive-definite lattice with the same Dynkin diagram.

For convenience, we record that for a sheaf $E$ of rank $r$, $$h^0(E)-h^1(E)+h^2(E)=\chi(E)=r+\frac{1}{2}c_1(E)^2-c_2(E).$$

The following two simple propositions will be especially useful:
\begin{Prop}[\cite{CD}]\label{effective} Let $D$ be a divisor with $D^2\geq 0$ and $D\ncong 0,K_Y$.  Then $D$ is effective or $-D$ is effective.  If $D$ is effecive, then $D+K_Y$ is also effective.
\end{Prop}
\begin{proof} As the proof is simple, we include it here.  First note that the effective cone is dual to the nef cone, so $D$ effective implies that $D+K_Y$ is effective as they are numerically equivalent, proving the final claim.  By Riemann-Roch, $$h^0(D)-h^1(D)+h^2(D)=\frac{1}{2}D^2+1\geq 1,$$ so $$h^0(D)+h^0(-D+K_Y)=h^0(D)+h^2(D)\geq 1.$$  Thus either $D$ is effective or $-D+K_Y$ is effective, but not both since then $K_Y$ would be effective.  If $-D+K_Y$ is effective, then $-D$ must be as well, hence the first claim.
\end{proof}

\begin{Def}\label{phi} For any divisor $D$ with $D^2>0$, we define $$\phi(D)=\inf\{|D\cdot F| |F\in \Pic(Y),F^2=0\}.$$
\end{Def}

The most important property of $\phi$ for us is the following \cite[Section 2.7]{CD}:
\begin{Thm}\label{phi} $0<\phi(D)^2\leq D^2$.
\end{Thm}

The importance of effective divisors of square zero in the geometry of Enriques surfaces lies partially in the fact that they correspond to elliptic pencils.  In particular, on an unnodal Enriques surface, a primitive effective divisor $F$ with $F^2=0$ is irreducible and $2F$ is one of two double fibers of the elliptic pencil $|2F|$.  Likewise, all complete elliptic pencils arise this way.  Such an $F$ is called a half-pencil.  The following final fact will also be of use to us:

\begin{Prop}[\cite{CD}]\label{pencil} For every elliptic pencil $|2E|$ on an Enriques surface $Y$, there exists an elliptic pencil $|2F|$ such that $E.F=1$.
\end{Prop}
\begin{proof} As $\Num(Y)$ is unimodular, we can find $F'\in\Pic(Y)$ such that $E.F'=1$.  Then since $\Num(Y)$ is even, $F:=F'-\frac{F'^2}{2}E\in\Pic(Y)$, $F^2=0$ and $F.E=1$, so $|2F|$ is the required elliptic pencil.
\end{proof}
\section{Classification of chern classes}

Now we can start in on the classification proper with the following prelimanary result:

\begin{Lem}\label{divisor} For any divisor $L$ on an unnodal Enriques surface $Y$, we can find a divisor $S$ such that $0\leq (L-4S)^2\leq 54$ and $L-4S>0$.
\end{Lem}
\begin{proof} Since $(L+4H)^2>0$ for some ample divisor, we can assume $L$ is effective and $L^2>0$.  Suppose $L^2\geq 64$, then by Theorem \ref{phi} we can find $F>0$ with $F^2=0$ such that $0<L.F\leq \sqrt{L^2}$, and in this case $0<8L.F\leq8\sqrt{L^2}\leq L^2$.  Now $(L-4F)^2=L^2-8L.F$, so $$0\leq (L-4F)^2<L^2.$$  If $56\leq L^2\leq 62$, then we can find $F>0$ with $F^2=0$ such that $0<L.F\leq 7$, so $0\leq (L-4F)^2<L^2$.  In both cases, $(L-4F).F=L.F>0$, so $L-4F>0$ by Proposition \ref{effective}.
\end{proof}

We use this to prove the following theorem:

\begin{Thm}\label{reduction} We can find a divisor $D$, depending only on $L\in \Pic(Y)$ and $m\in \Z$, such that for any rank 4 vector bundle $E$ with $c_1(E)=L$ and $c_2(E)=m$ on an unnodal Enriques surface $Y$, $$c_2(E(D))=\frac{1}{2}c_1(E(D))^2+k,$$ where $k\in\{-5,...,1\}$.  If $c_1(E(D))^2\geq 0$, then $c_1(E(D))>0$.
\end{Thm}
\begin{proof} First let us note that for any divisor $T$, $c_1(E(-T))=c_1(E)-4T$.  Then by Theorem \ref{phi} and Lemma \ref{divisor} above, we can find a divisor $T$ such that $0\leq c_1(E(-T))^2\leq 54$ and $c_1(E(-T))>0$.  

If $c_1(E(-T))^2>0$, then we can find $F>0$ with $F^2=0$ such that $0<c_1(E(-T)).F\leq 7$.  Then $$c_1(E(-T+nF))^2=(c_1(E(-T))+4nF)^2=c_1(E(-T))^2+8n(c_1(E(-T)).F),$$ and $$c_2(E(-T+nF))=c_2(E(-T))+3n(c_1(E(-T)).F).$$  Thus $$c_2(E(-T+nF))-\frac{1}{2}c_1(E(-T+nF))^2=c_2(E(-T))-\frac{1}{2}c_1(E(-T))^2-n(c_1(E(-T)).F).$$ Since $c_1(E(-T)).F=1,...,7$ we may find $n\in\Z$ which puts us in the range $-5,...,1$.  If $c_1(E(-T+nF))^2\geq 0$, then $$c_1(E(-T+nF)).F=(c_1(E(-T))+4nF).F=c_1(E(-T)).F>0,$$ so $c_1(E(-T+nF))$ is effective by Proposition \ref{effective}.

Now suppose $c_1(E(-T))^2=0$.  Then since $c_1(E(-T))>0$, $c_1(E(-T))=mH$ with $m>0$ such that $|2H|$ forms an elliptic pencil.  Take $F>0, F^2=0$ such that $F.H=1$.  If $4|m$, then $c_1(E(-T-\frac{m-4}{4}H))=4H$, so $c_1(E(-T-\frac{m-4}{4}H)).F=4$ and we're in one of the previous cases.  If $m\equiv j \mod 4$ with $j=1,2,3$, then $c_1(E(-T-\frac{m-j}{4}H)).F=j$, and we're in one of the above cases.  The resulting divisor is effective in any event.
\end{proof}

\begin{Rem} The above theorem indeed reduces the work to a finite number of cases.  It's entirely possible that a smaller number of cases is sufficient upon a closer look at the arithmetic of divisors on Enriques surfaces.
\end{Rem}

\section{Existence results in rank four}
\subsection{Numerical constraints} Let us summarize what the above section allows us to do.  By twisting by an appropriate line bundle, we can force $c_1$ and $c_2$ of a rank 4 sheaf to satisfy the relation $c_2-\frac{1}{2}c_1^2=k$ for $-5\leq k\leq 1$, making the Mukai vector $v=(4,c_1,2-k)$.  The existence of a stable coherent sheaf with these chern classes forces the dimension of the appropriate moduli space $M:=M_{H,Y}(v)$ to be non-negative.  Thus we get the following restrictions on $c_2$ and $c_1^2$:
\begin{equation}\label{table}
\begin{tabular}{c c c c}
$k$ & $\dim M$ & $\min c_2$ & $\min c_1^2$ \\ 
\hline
1&$2c_2-9$&5&8 \\
0&$2c_2-15$&8&16 \\
-1&$2c_2-21$&11&24 \\
-2&$2c_2-27$&14&32 \\
-3&$2c_2-33$&17&40 \\
-4&$2c_2-39$&20&48 \\
-5&$2c_2-45$&23&56 \\
\end{tabular}
.\end{equation}

Notice that by choosing $k$ in this range we force $c_1^2>0$, so by Theorem \ref{reduction} $c_1$ can be assumed to be effective and thus ample since $Y$ is unnodal.  Moreover, $v(E\otimes \mathcal O(D))=v(E)\cdot e^D$ is an isometry of the Mukai lattice, so it preserves the property of being primitive.  The assumption that our original Mukai vector was primitive allows us to use the usual lower bound $v^2+1$ as the dimension of $M$ above: coherent sheaves $E$ with $E\cong E\otimes \OO(K_Y)$ would have Mukai vector divisible by 2, as noted above.  In this section we prove the following existence result:

\begin{Thm}\label{four} Let $v=(4,c_1,2-k)$ be a primitive Mukai vector, with $-5\leq k\leq 1$ and $H:=c_1$ ample.  Assume that $c_1^2$ satisfies the inequality in (\ref{table}) for the corresponding $k$.  Then there exists a $\mu_H$-stable vector bundle with Mukai vector $v$ on an unnodal Enriques surface $Y$.
\end{Thm}

We will be modifying the argument of \cite[Theorem 5.1.6]{HL} for our purposes, but first we need the following lemma:
\begin{Lem}\label{CB} Let $L$ be a very ample divisor on a smooth projective surface $X$, and choose $1\leq n\leq h^0(X,L)$.  Then the locus $S\subset X^{[n]}$ parametrizing 0-dimensional subschemes $Z\subset X$ with length $l(Z)=n$ which fail to impose independent conditions on $H^0(X,L)$ is a divisor.  Moreover, its generic element consists of those $Z$ comprised of $n$ distinct points, $n-1$ of which impose independent conditions on $H^0(X,L)$.
\end{Lem}
\begin{proof}  Indeed, choose a subspace $W\subset H^0(X,L)$ of dimension $n$.  Then the subscheme $S$ can be described as the locus where the natural map $$W\otimes \OO_{X^{[n]}}\rightarrow {\pi_1}_*(\pi_2^*\OO(L)\otimes \mathcal O_{\mathcal Z_n})$$ of vector bundles of rank $n$ fails to be an isomorphism, where $\mathcal Z_{n}\subset X^{[n]}\times X$ is the universal subscheme and $\pi_1,\pi_2$ are the respective projections.  Thus $S$ has codimension at most 1.  Since $L$ is very ample, the generic element of $X^{[n]}$ imposes independent linear conditions, so $S$ has codimension exactly 1. 

The last claim is clear.
\end{proof}

Finally, we must also recall the definition of the Cayley-Bacharach property:
\begin{Def} A pair $(L,Z)$ of a line bundle $L$ and a 0-dimensional subscheme $Z\subset X$ on a smooth projective variety $X$  is said to satisfy the \emph{Cayley-Bacharach property} if for every subscheme $Z'\subset Z$ with $l(Z')=l(Z)-1$ and every $s\in H^0(X,L)$ with $s|_{Z'}=0$, $s|_Z=0$ as well.  In other words, $Z$ fails to impose independent linear conditions on $H^0(X,L)$.
\end{Def}

Now we may prove the above theorem:

\begin{proof}(of Theorem \ref{four}) Let us consider extensions of the form $$0\rightarrow \mathcal O(-2H)\rightarrow E\rightarrow \bigoplus_{i=1}^3 I_{Z_i}(H)\rightarrow 0,$$ where the $Z_i$ are disjoint reduced subschemes of dimension zero.  Then, from the proof of \cite[Theorem 5.1.6]{HL}, we have that if the pair $(\mathcal O(3H+K_Y),Z_i)$ satisfies the Cayley-Bacharach property for each $i$, then a locally free extension as above exists.  To ensure that the Cayley-Bacharach property holds for each $Z_i$,  it follows from Lemma \ref{CB} that it suffices to let $Z_i$ be the generic element of the divisor $S_i\subset Y^{[l(Z_i)]}$ as described in the last statement of the lemma, where $l(Z_i)\leq h^0(3H+K_Y)=\frac{9}{2}H^2+1$.  Note that $3H+K_Y=3(H+K_Y)$ is indeed very ample \cite[Corollary 4.10.2]{CD}.  Moreover, for $E$ obtained as above to satisfy the required relation between $c_1(E)$ and $c_2(E)$, we must have $\frac{1}{2}H^2+k=c_2(E)=\sum_{i=1}^3 l(Z_i)-3H^2$.  It is easily checked that we can choose $l(Z_i)\leq \frac{9}{2}H^2+1$ and satisfying this equation.  

Suppose we have chosen $Z_i$ as above, and take a $\mu_H$-unstable locally free extension as above.  Let $E'=E(2H)$ (still unstable), and consider a destabilizing locally free sheaf $F\subset E'$ of rank $0<s<4$.  It suffices to consider locally free destabilizing subobjects since the reflexive hull of such objects will also destabilize and necessarily be locally free.  Since $\mu(F)\geq\mu(E')=\frac{9H^2}{4}>0$, $F$ must be contained in $\bigoplus_{i=1}^3 I_{Z_i}(3H)$, and passing to exterior powers gives a nonzero, and thus injective, homomorphism $$
\det(F)\otimes \mathcal O(-3sH)\rightarrow \bigoplus_{1\leq i_1<...<i_s\leq 3} I_{Z_{i_1}\cup ...\cup Z_{i_s}}.$$  Thus there is an effective divisor $D$ of degree $$\deg(D)=D.H=s\mu(\mathcal O(3H))-s\mu(F)\leq \frac{3s}{4}H^2,$$ which contains at least $s$ of the 3 subschemes $Z_i$.  Since the divisor is effective with bounded degree, there are only finitely many such divisor classes $[D]$.  For each such divisor class, the Hilbert scheme of linearly equivalent effective divisors in this class is $|D|:=\P(H^0(\mathcal O(D)))$ of dimension $\frac{1}{2}D^2$ if $D^2>0$, 1 if $D$ is an elliptic pencil, or 0 if $D$ is an elliptic half-pencil (both of these are when $D^2=0$).  

Let $\Hilb_{D,i_1,...,i_s}$ be the flag Hilbert scheme of pairs $[Z_{i_1}\cup ...\cup Z_{i_s} \subset D']$ where $D'\in |D|$ and contains $Z_{i_1}\cup ...\cup Z_{i_s}$ with $[Z_{i_j}]\in Y^{[l(Z_{i_j})]}$ for each $j$.  It comes equipped with two projections $p_1:\fhilb\to |D|$ and $p_2:\fhilb\to \prod_{j=1}^s Y^{[l(Z_{i_j})]}$ which associate to a flag $[Z_{i_1}\cup ...\cup Z_{i_s} \subset D']$ the effective divisor $D'$ or the $s$-tuple of 0-dimensional subschemes $(Z_{i_1},...,Z_{i_s})$, respectively. For each $D'\in|D|$, the fibre of the morphism $p_1:\fhilb\rightarrow |D|$ has dimension $\sum_{j=1}^s l(Z_{i_j})$, so $\dim \fhilb=\frac{1}{2}D^2+\sum_{j=1}^s l(Z_{i_j})$, $1+\sum_{j=1}^s l(Z_{i_j})$, or $\sum_{j=1}^s l(Z_{i_j})$, respectively.  Thus the image of $\fhilb$ under the natural projection $p_2:\fhilb\rightarrow\prod_{j=1}^s Y^{[l(Z_{i_j})]}$ has dimension at most $\frac{1}{2}D^2+\sum_{j=1}^s l(Z_{i_j}),1+\sum_{j=1}^s l(Z_{i_j})$, or $\sum_{j=1}^s l(Z_{i_j})$.  Now $(D.H)^2\leq \frac{9s^2}{16}(H^2)^2$, so by the Hodge Index Theorem $D^2H^2\leq (D.H)^2$ and thus $D^2\leq \frac{9s^2}{16}H^2$.  Thus the image of $\fhilb$ in $\prod_{j=1}^s Y^{[l(Z_{i_j})]}$ has dimension at most $\frac{9s^2}{32}H^2+\sum_{j=1}^s l(Z_{i_j})$, $1+\sum_{j=1}^s l(Z_{i_j})$, or $\sum_{j=1}^s l(Z_{i_j})$, respectively.  The dimension of $\prod_{j=1}^s Y^{[l(Z_{i_j})]}$ is $2\sum_{j=1}^s l(Z_{i_j})$.  

We would like to choose $(Z_{i_1},...,Z_{i_s})\in\prod_{j=1}^s S_j$ (to ensure the freeness of $E$) but not in the image of $\fhilb$ (to ensure its stable).  It suffiices to show that $\dim p_2(\fhilb)<2\sum_{j=1}^s l(Z_{i_j})-s$, i.e. $\frac{9s^2}{32}H^2+\sum_{j=1}^s l(Z_{i_j})<2\sum_{j=1}^s l(Z_{i_j})-s$ or equivalently $\frac{9s^2}{32}H^2<\sum_{j=1}^s l(Z_{i_j})-s$.  Suppose we choose all the lengths to be roughly equal, say $l(Z_1)=t,l(Z_2)=t-1,l(Z_3)=t-a$, where $t\leq \frac{9}{2}H^2+1$.  Then $3t-1-a=\sum_{i=1}^3 l(Z_i)=\frac{7}{2}H^2+k$, so $t=\frac{7}{6}H^2+\frac{k}{3}+\frac{a+1}{3}$.  Here $a$ is chosen as small as possible to make $t$ an integer, so it follows that $a\equiv 2+2k+H^2(\mod 3)$ and thus $a+1\leq 2+H^2$ if $k=1,-2,-5$, $a+1\leq H^2$ if $k=0,-3$, and $a+1\leq 1+H^2$ if $k=-1,-4$.

Now let us consider the three possibilities for $s$.  First, if $s=3$, then $\sum_{i=1}^3 l(Z_{i})=\frac{7}{2}H^2+k$, so we're requiring that $-\frac{31}{32}H^2<k-3$, which is always satisfied.  Indeed, this follows upon inspection of the minimum $H^2$ values from (\ref{table}).  

If $s=2$, then $\sum_{j=1}^2 l(Z_{i_j})=\frac{7}{2}H^2+k-l(Z_i)$, where $i\neq i_1,i_2$, so we're requiring that $-\frac{19}{8}H^2<k-2-l(Z_i)$.  By construction $l(Z_i)\leq t=\frac{7}{6}H^2+\frac{k}{3}+\frac{a+1}{3}$ so $k-2-l(Z_i)\geq -\frac{7}{6}H^2-\frac{k}{3}-\frac{a+1}{3}+k-2=-\frac{7}{6}H^2+\frac{2}{3}k-\frac{a+1}{3}-2$.  If $k=1,-2,-5$, then $-\frac{7}{6}H^2+\frac{2}{3}k-\frac{a+1}{3}-2\geq -\frac{9}{6}H^2+\frac{2}{3}k-\frac{8}{3}$.  Inspecting the table above, it follows that $k>4-\frac{21}{16}H^2$ for these values of $k$ and all possible values of $H^2$, and from this it follows that $-\frac{9}{6}H^2+\frac{2}{3}k-\frac{8}{3}>-\frac{19}{8}H^2$, as required.  For $k=0,3$, we have $k>3-\frac{21}{16}H^2$ for all possible values of $H^2$, and thus $-\frac{7}{6}H^2+\frac{2}{3}k-\frac{a+1}{3}-2\geq -\frac{9}{6}H^2+\frac{2}{3}k-2>-\frac{19}{8}H^2$.  Finally, if $k=-1,-4$, then $k>\frac{7}{2}-\frac{21}{16}H^2$ for all possible values of $H^2$, and it follows that $-\frac{7}{6}H^2+\frac{2}{3}k-\frac{a+1}{3}-2\geq -\frac{9}{6}H^2+\frac{2}{3}k-\frac{7}{3}>-\frac{19}{8}H^2$, as required.

For the last case, when $s=1$, we want $\frac{9}{32}H^2<l(Z_i)-1$.  Notice that for each $i$ $$l(Z_i)\geq t-a=\frac{7}{6}H^2+\frac{k}{3}-\frac{2}{3}a+\frac{1}{3}\geq \frac{1}{2}H^2+\frac{k}{3}+\left \{\begin{array}{l}  -1\mbox{ if }k\equiv 0(\mod 3)\\ -\frac{1}{3}\mbox{ if }k\equiv 1(\mod 3)\\\frac{1}{3}\mbox{ if }k\equiv 2(\mod 3)\end{array}\right\}.$$  It is easily checked from (\ref{table}) that this number is always greater than $\frac{9}{32}H^2+1$, as required.

It follows from all of the above that we can choose the subschemes $Z_i$ to ensure that the extension $E$ is locally free and $\mu_H$-stable, as required.
\end{proof}

\section{$M_{H,Y}(mv_0)\neq\varnothing$ for $v_0^2\geq -1$}
In this section, we prove the non-emptiness of moduli spaces of Gieseker-semistable sheaves with given Mukai vector on an unnodal Enriques surface with respect to a generic polarization.  We recall the most important consequence of this genericity assumption: the only sheaves with the same reduced Hilbert polynomial as a given sheaf of Mukai vector $v$ have Mukai vectors on the ray $\R_{>0}v$.  Let us set notation and denote by $v_0=(r,c_1,s)$ a primitive Mukai vector of positive rank with $v_0^2\geq -1$.  We further set $H$ to be a polarization that is generic with respect to $v_0$.  To prove our main theorem, completing the sequence of results initiated by Kim in \cite{Kim98,Kim,Kim06} and Hauzer in \cite{Hau10}, we first recall the definition of the virtual Hodge polynomial of a variety $X$ over $\C$.  The cohomology with compact support $H^*_c(X,\Q)$ has a natural mixed Hodge structure, and we define $e^{p,q}(X):=\sum_k (-1)^kh^{p,q}(H_c^k(X,\Q))$ and $e(X):=\sum_{p,q}e^{p,q}(X)x^py^q$ to be the virtual Hodge number and virtual Hodge polynomial, respectively.  Our main theorem is as follows:

\begin{Thm}\label{existence} For generic polarization $H$, the moduli space of Gieseker-semistable sheaves of Mukai vector $v=mv_0$ with respect to $H$ on an unnodal Enriques surface $Y$, $M_{H,Y}(v)$, is non-empty.
\end{Thm}
\begin{proof}  First note that it suffices to prove the result for $m=1$, i.e. $v$ primitive.  Indeed, for any $E_0\in M_{H,Y}(v_0)$, $E^{\oplus m}$ is semistable of Mukai vector $v=mv_0$.  So we may in-fact assume that any Gieseker-semistable sheaf is stable as well.

We can further restrict ourselves by recalling that non-emptiness was proven in \cite[Theorem 4.6]{Yos03} for $r$ odd.  

So we are reduced to proving the theorem when $r$ is even and positive.  Then \cite[Theorem 2.8]{Hau10} shows that the virtual Hodge polynomial of $M_{H,Y}(v)$ is equal to the virtual Hodge polynomial of a moduli space of sheaves on $Y$ with primitive Mukai vector of rank 2 or 4 that are Gieseker-stable with respect to $H$.  

In the rank 2 case, we may use the main theorem of \cite{Kim98} to reduce to the case $s=1+\frac{1}{2}c_1^2-c_2=1-k$, where $k=0,1$.  The existence of rank 2 $\mu_{c_1}$-stable vector bundles with Mukai vector $v$ is proven in \cite{Kim06} for $v^2>0$ and $k=1$, $v^2\geq 0$ and $k=0$, respectively, where these numerical restrictions guarantee the ampleness of $c_1$.  If $c_1$ is not generic with respect to $v$, it must lie on a wall.  But then $\mu_{c_1}$-stable sheaves remain $\mu_H$-stable for $H$ generic on one side of the wall (see \cite[Section 4.C]{HL}), hence the claim.  The remaining case in rank 2 is for $v^2=0$ and $k=1$, i.e. $v=(2,c_1,0)$.  This is the content of \cite[Theorem 0.1]{Hau10}.

In the rank 4 case, we may use Theorem \ref{reduction} above to reduce to the case $s=2+\frac{1}{2}c_1^2-c_2=2-k$, where now $k=-5,...,1$.  Then Theorem \ref{four} shows the existence of $\mu_{c_1}$-stable vector bundles of rank 4 and Mukai vector $v$.  Repeating the argument above if $c_1$ is not generic, we deduce the non-emptiness in this case as well and with it the theorem.
\end{proof}
\section{The Existence of Stable Objects}
Having shown in the previous section that for $v=mv_0$ with $v_0$ primitive, $v_0^2\geq -1$, and $H$ generic with respect to $v$, $M_{H,Y}(v)\neq\varnothing$, we turn now to the question of the non-emptiness of $M^s_{H,Y}(v)$, the open subset parametrizing stable sheaves.  As noted above, $M_{H,Y}(v)=M^s_{H,Y}(v)$ in case $m=1$, and for $m>1$, any destabilizing subsheaf of a sheaf in $M_{H,Y}(v)$ must have Mukai vector $m'v_0$ for $m'<m$.  We generalize standard arguments to show the following result:
\begin{Thm}\label{num dim} Let $v=mv_0$ be a Mukai vector with $v_0$ primitive and $m>0$ with $H$ generic with respect to $v$.  
\begin{enumerate}
\item The moduli space of Gieseker-semistable sheaves $M_{H,Y}(v)\neq \varnothing$ if and only if $v_0^2\geq -1$.
\item Either $\dim M_{H,Y}(v)=v^2+1$ and $M^s_{H,Y}(v)\neq\varnothing$,  or $m>1$ and $v_0^2\leq 0$.
\item If $M_{H,Y}(v)\neq M^s_{H,Y}(v)$ and $M^s_{H,Y}(v)\neq\varnothing$, the codimension of the semistable locus is at least 2 if and only if $v_0^2>1$ or $m>2$.  Moreover, in this case $M_{H,Y}(v)$ is normal with torsion canonical divisor.
\end{enumerate}
\end{Thm}
\begin{proof} If $v_0^2\geq -1$, then part (a) follows from Theorem \ref{existence} above.  For the converse, note that any stable factor of an element of $M_{H,Y}(v)\neq \varnothing$ would have to have Mukai vector $m'v_0$ for $m'<m$ by genericity of $H$.  But then $m'^2v_0^2=(m'v_0)^2\geq -1$, so $v_0^2\geq -1$.

For (b), again notice that the genericity of $H$ means that any stable factor of an object of $M_{H,Y}(v)$ must have Mukai vector $m'v_0$ for $m'<m$, which implies that the strictly semistable locus is the image of the natural map $$\SSL: \coprod_{m_1+m_2=m,m_i>0} M_{H,Y}(m_1v_0)\times M_{H,Y}(m_2v_0)\rightarrow M_{H,Y}(v).$$  Assume $v_0^2>0$.  Then for $m=1$, $M_{H,Y}(v)=M^s_{H,Y}(v)$, and we have noted already that $\dim M_{H,Y}(v)=v^2+1$.  If $m>1$, then by induction, we deduce that the image of the map $\SSL$ has dimension equal to the maximum of $(m_1^2+m_2^2)v_0^2+2$ for $m_1+m_2=m,m_i>0$.  This is strictly less than $v^2+1$.

Furthermore, we can construct a semistable sheaf $E'$ with Mukai vector $v$ which is also Schur, i.e. $\Hom(E',E')=\C$.  By the inductive assumption, we can consider $E\in M_{H,Y}^s((m-1)v_0)$, and let $F\in M_{H,Y}(v_0)$.    Now $\chi(F,E)=-(v(F),v(E))=-(m-1)v_0^2<0$, so $\text{Ext}^1(F,E)\neq 0$.  Take $E'$ to be a nontrivial extension $$0\rightarrow E\rightarrow E'\rightarrow F\rightarrow 0.$$  Then any endomorphism of $E'$ gives rise to a homomorphism $E\rightarrow F$, of which there are none since these are both stable of the same phase and have different Mukai vectors (or can be chosen to be non-isomorphic if $m=2$).  Thus any endomorphism of $E'$ induces an endomorphism of $E$, and the kernel of this induced map $\Hom(E',E')\rightarrow \Hom(E,E)=\C$ is precisely $\Hom(F,E')$, which vanishes since the extension is non-trivial.  Thus $\Hom(E',E')=\C$.

We can deduce non-emptiness of $M^s_{H,Y}(v)$ from a dimension estimate as follows.  Since $E'$ is Schur, we get $$v^2+1\leq \dim_{E'} M_{H,Y}(v)\leq \dim T_{E'}M_{H,Y}(v)=v^2+1+\hom(E',E'\otimes\OO(K_Y)).$$  As we mentioned above, the strictly semistable locus must have dimension smaller than $v^2+1$.  So even though $E'$ is not stable, it lies on a component which must contain stable objects.  Moreover, as noted in \cite{Kim}, (smooth) components of the stable locus of dimension greater than $v^2+1$ can occur only if $v_0^2=0$, so $v_0^2>0$ implies that the locus of points fixed by $-\otimes\OO(K_Y)$ has positive codimension.  Then we may choose $E\in M_{H,Y}^s((m-1)v_0)$ so that $E\ncong E\otimes\OO(K_Y)$.  Stability of $E$ and $F$ and a diagram chase then show that $\Hom(E',E'\otimes\OO(K_Y))=0$, so $M_{H,Y}(v)$ is smooth at $E'$ of dimension $v^2+1$ as claimed.

Furthermore, observe that the strictly semistable locus has codimension $$v^2+1-(m_1^2v_0^2+m_2^2v_0^2+2)=(m_1+m_2)^2v_0^2+1-(m_1^2v_0^2+m_2^2v_0^2+2)=2m_1m_2v_0^2-1\geq 2,$$ if $v_0^2>1$ or $m>2$, hence the first part of (c).  For the second part of (c), notice that the singularities of $M^s_{H,Y}(v)$, i.e. where $\ext^2(E,E)=1$, are all hypersurface singularities, so normality follows from the dimension estimates in Theorem \ref{Kim} and the large codimension of the strictly semistable locus.  $K$-triviality follows from these considerations and the proof of \cite[Proposition 8.3.1]{HL}.

If $ v_0^2\leq 0$, then it is easily seen that stable sheaves occur only if $m=1$ and $\dim M_{H,Y}(v)=v^2+1$ in this case, while $$\dim M_{H,Y}(v)= \left \{\begin{array}{l}  0\mbox{ if }v_0^2=-1\\ m\mbox{ if }v_0^2=0\end{array}\right\}\neq v^2+1,$$ if $m>1$.
\end{proof}

\section{Irreducibility}

As noted in the introduction, irreducibility of Gieseker moduli spaces on unnodal Enriques surfaces remains an open question.  Yoshioka has resolved this for primitive $v$ of odd rank in \cite{Yos03} by showing that $M_{H,Y}(v)$ has two irreducible components distinguished by their determinant line bundles which differ by $K_Y$.  In this paper, we restrict ourselves to proving the following result:
\begin{Thm} Let $v$ be a primitive Mukai vector on an unnodal Enriques surface $Y$ with $v^2=0$.  Then $M_{H,Y}(v,L)$ is irreducible, where $L$ is one of the two lines bundles with $c_1(L)=c_1(v)$.
\end{Thm}
\begin{proof} We follow the idea of \cite[Theorem 6.1.8]{HL}.  Since $v$ is primitive and $H$ generic, any $H$ Gieseker-semistable sheaf is stable, so $M:=M_{H,Y}(v)$ is projective and smooth, and any connected component is thus irreducible and smooth.  

Let $M_1$ be one of the connected components of $M$, and fix a quasi-universal family $\mathcal E$ of similitude $s=\rk(\mathcal E)/r$ over $M_1\times Y$ with projections $p:M_1\times Y\to M_1,q:M_1\times Y\to Y$.  Denote by $M_1(K_Y)$ the (possibly disjoint) connected component consisting of objects $F(K_Y)$ for $F\in M_1$.  Let $[F]\in M$ be an arbitrary point in the moduli space.  For any $t\in M_1$ we have $$\Hom(F,\mathcal E_t)=\left \{\begin{array}{l}  0\mbox{ if }F^{\oplus s}\ncong \mathcal E_t\\ \C^s\mbox{ if }F^{\oplus s}\cong \mathcal E_t\end{array}\right\}$$ and $$\Ext^2(F,\mathcal E_t)\cong\Hom(\mathcal E_t,F(K_Y))^{\vee}=\left \{\begin{array}{l}  0\mbox{ if }F(K_Y)^{\oplus s}\ncong \mathcal E_t\\ \C^s\mbox{ if }F(K_Y)^{\oplus s}\cong \mathcal E_t\end{array}\right\}.$$  Since $\chi(F,\mathcal E_t)=s\chi(F,F)=-sv^2=0$, we also have $$\Ext^1(F,\mathcal E_t)=\left \{\begin{array}{l}  0\mbox{ if }F^{\oplus s},F(K_Y)^{\oplus s}\ncong \mathcal E_t\\ \C^s\mbox{ if }F^{\oplus s}\cong \mathcal E_t\\ \C^s\mbox{ if }F(K_Y)^{\oplus s}\cong \mathcal E_t\end{array}\right\},$$ where we note that since $v$ is primitive $F\ncong F(K_Y)$.  Thus if $[F]\notin M_1$ or $M_1(K_Y)$, then $\Ext^i(F,\mathcal E_t)=0$ for all $i$ and all $t\in M_1$, so $\lExt_p^i(q^*F,\mathcal E)=0$ for all $i$.

For $[F]\in M_1$ or $M_1(K_Y)$, we have to work a little harder.  Suppose first that $M_1=M_1(K_Y)$.  By \cite{BPS} there exists a complex $\mathcal P^{\bullet}$ of locally free sheaves $\mathcal P^i$ of finite rank such that the $i$-th cohomology $\mathcal H^i(\mathcal P^{\bullet})\cong \lExt_p^i(q^*F,\mathcal E)$ and $\mathcal H^i(\mathcal P^{\bullet}_t)\cong \Ext^i(F,\mathcal E_t)$.  Moreover, this complex is bounded from above and we may assume $\mathcal P^i=0$ for $i<0$.  Since $\Ext^i(F,\mathcal E_t)=0$ for $i>2$, the complex is exact at $\mathcal P^i$ for $i>2$.  Of course, $\lExt_p^0(q^*F,\mathcal E)$ is a skyscraper sheaf concentrated at $t_0=[F]$.  Since it's a subsheaf of the locally free sheaf $\mathcal P^0$, it must be 0.  Furthermore, $\ker(d^i)$ is always locally free as the kernel of a surjection from the locally free $\PP^i$ to the torsion-free sheaf $\im(d^i)$ on the smooth curve $M_1$.  As $\mathcal P^{\bullet}$ is exact at $\mathcal P^i$ for $i>2$ and $\mathcal P^i=0$ for large $i$, we can work backwards using exactness and  replace $\mathcal P^2$ by $\ker(d^2)$ to get $$0\rightarrow \im(d^1)\rightarrow \mathcal P^2\rightarrow \mathcal H^2(\mathcal P^{\bullet})\rightarrow 0.$$  It follows that $$\mathcal H^2(\mathcal P^{\bullet})(t)=\mathcal P^2(t)/\im(d^1)(t)\cong \mathcal H^2(\mathcal P^{\bullet}(t))\cong \text{Ext}^2(F,\mathcal E_t)$$ for any closed point $t$.  Thus $\lExt_p^2(q^*F,\mathcal E)$ is a torsion sheaf concentrated at $t_1=[F(K)]$ with length $s$.  Finally, we see that $\lExt^1_p(q^*F,\EE)$ is a torsion sheaf concentrated at $t_0$ and $t_1$ with length $s$ at each point.

If instead $M_1\neq M_1(K_Y)$, then the same considerations show that $\lExt^0_p(q^*F,\EE)=0$ in any case, and $\lExt^2_p(q^*F,\EE)=0$ and $\lExt^1_p(q^*F,\EE)=\C(t_0)^{\oplus s}$ if $F\in M_1$ and $\lExt^2_p(q^*F,\EE)=\C(t_0)^{\oplus s}$ and $\lExt^1_p(q^*F,\EE)=\C(t_0)^{\oplus s}$ if $F\in M_1(K_Y)$, respectively.

By the Grothendieck-Riemann-Roch formula, $$a:=ch([\lExt_p^0(q^*F,\mathcal E)]-[\lExt_p^1(q^*F,\mathcal E)]+[\lExt_p^2(q^*F,\mathcal E)])$$ depends only only $\ch(F)$ and $\ch(\mathcal E)$.  Since $\ch(F)$ is constant for all $[F]\in M$, so is $a$.  For $F\notin M_1$ or $M_1(K_Y)$, we get $a=0$.  Likewise, if $M_1\neq M_1(K_Y)$, then we get $a=0$ if $F\in M_1(K_Y)$.  If $F\in M_1$, then $a=-ch(\C(t_0))$ in either case, and $0\neq -\chi(\C(t_0))=\langle a\cdot td(M),[M_1]\rangle$, which is a contradiction unless $M$ is irreducible and equal to $M_1$.
\end{proof}

\begin{Rem} Note that the above proof shows that if $E\in M_{H,Y}(v,L)$ with $v$ as above, then $M_{H,Y}(v,L(K_Y))=\varnothing$. Moreover, we see that tensoring with $K_Y$ acts as an involution on the single irreducible component of $M_{H,Y}(v)$.  We believe this to be the case in general for primitive $v$ of even rank.
\end{Rem}

\end{document}